\documentclass[12pt]{amsart}
\usepackage{amscd}
\usepackage{amsmath}
\usepackage{amssymb}

\usepackage[all]{xy}
\def\NZQ{\Bbb}               
\def\NN{{\NZQ N}}

\def\ZZ{{\NZQ Z}}

\def\CC{{\NZQ C}}

\def\PP{{\NZQ P}}

\def\frk{\frak}               

\def\Phi{{\frk n}}
\def\Phi{{\frk N}}
\def\opn#1#2{\def#1{\operatorname{#2}}} 
\opn\chara{char} \opn\length{\ell} \opn\pd{pd} \opn\rk{rk}
\opn\projdim{proj\,dim} \opn\injdim{inj\,dim} \opn\rank{rank}
\opn\depth{depth} \opn\sdepth{sdepth} \opn\fdepth{fdepth}
\opn\grade{grade} \opn\height{height} \opn\embdim{emb\,dim}
\opn\codim{codim}  \opn\min{min} \opn\max{max}

\opn\Tr{Tr} \opn\bigrank{big\,rank}
\opn\superheight{superheight}\opn\lcm{lcm}
\opn\trdeg{tr\,deg}
\opn\reg{reg} \opn\lreg{lreg} \opn\ini{in} \opn\lpd{lpd}
\opn\size{size}
\opn\div{div} \opn\Div{Div} \opn\cl{cl} \opn\Cl{Cl}
\opn\Spec{Spec} \opn\Supp{Supp} \opn\supp{supp} \opn\Sing{Sing}
\opn\Ass{Ass} \opn\Min{Min}
\opn\Ann{Ann} \opn\Rad{Rad} \opn\Soc{Soc}
\opn\Im{Im} \opn\Ker{Ker} \opn\Coker{Coker} \opn\Am{Am}
\opn\Hom{Hom} \opn\Tor{Tor} \opn\Ext{Ext} \opn\End{End}
\opn\Aut{Aut} \opn\id{id}  \opn\deg{deg}

\opn\nat{nat}
\opn\pff{pf}
\opn\Pf{Pf} \opn\GL{GL} \opn\SL{SL} \opn\mod{mod} \opn\ord{ord}
\opn\Gin{Gin} \opn\Hilb{Hilb}
\opn\aff{aff} \opn\con{conv} \opn\relint{relint} \opn\st{st}
\opn\lk{lk} \opn\cn{cn} \opn\core{core} \opn\vol{vol}
\opn\link{link} \opn\star{star}
\opn\gr{gr}

\def\pot#1#2{#1[\kern-0.28ex[#2]\kern-0.28ex]}

\opn\dirlim{\underrightarrow{\lim}}
\opn\inivlim{\underleftarrow{\lim}}
\let\to=\rightarrow

\def\Implies{\ifmmode\Longrightarrow \else
        \unskip${}\Longrightarrow{}$\ignorespaces\fi}
\def\implies{\ifmmode\Rightarrow \else
        \unskip${}\Rightarrow{}$\ignorespaces\fi}
\def\iff{\ifmmode\Longleftrightarrow \else
        \unskip${}\Longleftrightarrow{}$\ignorespaces\fi}

\let\:=\colon
\newtheorem{Theorem}{Theorem}[section]
\newtheorem{Lemma}[Theorem]{Lemma}
\newtheorem{Corollary}[Theorem]{Corollary}
\newtheorem{Proposition}[Theorem]{Proposition}
\newtheorem{Remark}[Theorem]{Remark}

\newtheorem{Example}[Theorem]{Example}

\newtheorem{Conjecture}[Theorem]{Conjecture}

\let\epsilon\varepsilon
\let\phi=\varphi
\let\kappa=\varkappa
\def\qed{\ifhmode\textqed\fi
      \ifmmode\ifinner\quad\qedsymbol\else\dispqed\fi\fi}
\def\textqed{\unskip\nobreak\penalty50
       \hskip2em\hbox{}\nobreak\hfil\qedsymbol
       \parfillskip=0pt \finalhyphendemerits=0}
\def\dispqed{\rlap{\qquad\qedsymbol}}

\opn\dis{dis}
\def\pnt{{\raise0.5mm\hbox{\large\bf.}}}

\opn\Lex{Lex}

\begin{document}

\title{\bf Hilbert series and Lefschetz properties of dimension one almost complete intersections}

\author{Alexandru Dimca and  Dorin Popescu}

\thanks{The  support of both authors from Institut Universitaire de France and the support of the second author  from  grant ID-PCE-2011-1023 of Romanian Ministry of Education, Research and Innovation are gratefully acknowledged.}

\address{Alexandru Dimca,  Univ. Nice Sophia Antipolis, CNRS,  LJAD, UMR 7351, 06100 Nice, France.}
\email{dimca@unice.fr}

\address{Dorin Popescu,  Simion Stoilow Institute of Mathematics of Romanian Academy, Research unit 5,
 P.O.Box 1-764, Bucharest 014700, Romania}
\email{dorin.popescu@imar.ro}

\maketitle
\begin{abstract} We generalize some properties related to Hilbert series and Lefschetz properties of Milnor algebras of projective hypersurfaces with isolated singularities to the more general case of an almost complete intersection ideal $J$ of dimension one. When the saturation $I$ of $J$ is a complete intersection, we get explicit formulas for a number of related invariants. New examples of hypersurfaces $V:f=0$ in $\PP^n$ whose Jacobian ideal $J_f$ satisfies this property  and with explicit nontrivial Alexander polynomials  are given in any dimension. A Lefschetz type property for the graded quotient $I/J$ is proved for $n=2$ and a counterexample due to A. Conca is given for such a property when $n=3$. Two conjectures are also stated in the paper.\\


\end{abstract}

\section*{Introduction}

Let $S=K[x_0,\ldots,x_n]$ be the polynomial algebra over a characteristic zero field $K$, ${\bf f}=\{f_0,\ldots,f_n\}$ be a system of $n+1$ homogeneous polynomials of $S$ with $\deg f_i=d_i$ and $J=({\bf f})$ the ideal in $S$ spanned by these polynomials. Suppose that $\dim S/J=1$ and $f_1,\ldots,f_n$ is a regular sequence in $S$. 

Let $V\subset {\PP}^n$ be a projective hypersurface defined by a homogeneous polynomial $f\in S$ of degree $d$. When ${\bf f}=\{f_0,\ldots,f_n\}$ is the set of its partial derivatives,  $J_f=(f_0,\ldots,f_n)$ is called the Jacobian ideal of $f$  and  $M(f)=S/J_f$ is the Milnor (or Jacobian) algebra of $f$.
Assume that $V$ has only isolated singularities, i.e.  $\dim S/J_f=1$. Then, maybe after a coordinate change, we can assume that $\{f_1,\ldots f_n\}$ is a regular sequence, i.e. this geometric setting provides the main example of the algebraic situation described above, with $d_i=d-1$ for all $i=0,...,n$.

The Hilbert series of the Milnor algebra $M(f)$ and a number of related invariants have been studied in \cite{CD} and a number of recent papers, see \cite{D}, \cite{DS}, \cite{DSt},  \cite{EM}, \cite{Se}, \cite{SW}.
In this paper we extend some of the results obtained for the Milnor algebra to the more general case of an almost complete intersection of dimension one $S/J$ as defined above. In some cases, this leads to simplifications of the proofs given in the special case of Milnor algebras.
As an example, compare the proof of Theorem \ref{main} to that of Proposition \ref{CIcase} given in \cite{CD} or consider Remark \ref{rk2}.
In addition, we give new evidence for a conjecture stated already in \cite{D}, see Conjecture \ref{C1} and Theorem \ref{comput}, (3).
When the saturation of $J$ is a complete intersection, we get explicit formulas for a number of related invariants, see Theorem \ref{comput}. New examples of hypersurfaces $V:f=0$ whose Jacobian ideal $J_f$ satisfies this property  and with explicit nontrivial Alexander polynomials  are given in any dimension, see Example \ref{E1}.

For a generic linear form $l \in S_1$, the structure of $S/J$ as a $K[l]$-module is closely related to various Lefschetz type properties, a subject investigated in the second and the forth sections, see for instance Theorem \ref{m1}  and   Theorem \ref{mainL}. We prove a  weak Lefschetz property of $S/J$ (that is in dimension one) similar to the existing one in the Artinian frame (see \cite{HMNW}, \cite{BK}), the  case $n=1$ being done in our Corollary \ref{n=1}. In the forth section a Lefschetz type property for the graded quotient $I/J$ is proved for $n=2$ and a counter example due to A. Conca is given for such a property when $n=3$, see Theorem \ref{mainL} and Remark \ref{rkL}. This implies a Lefschetz type property for the Milnor algebra $M(f)$  for $n=2$ holding in many cases, see Corollary \ref{corL} and Remark \ref{rkL0}.

For the interested readers, we point out that the duality properties of the quotient $I/J$ are used, but not discussed below, and can be found in \cite{DS}, \cite{Se} (see especially the third section) and in \cite{EM}.

A major challenge would be to find algebraic proofs for the properties of the Hilbert series of
the Milnor algebra $M(f)$ which depend on the type of the singularities of $V$ (e.g. properties holding for nodal hypersurfaces, characterized by the fact that they are the hypersurfaces with isolated singularities satisfying $I=\sqrt { J_f}$). Such results have been obtained in \cite{DS}, \cite{DSt3}, \cite{DSt} using rather advanced Hodge theory.

Two conjectures are also stated the paper, see Conjectures \ref{C1} and \ref{C2}.

\section{Hilbert series of dimension one almost complete intersections}

Let $S=K[x_0,\ldots,x_n]$ be the polynomial algebra over a field $K$, $f=\{f_0,\ldots,f_n\}$ be a system of $n+1$ homogeneous polynomials of $S$ with $\deg f_i=d_i$ and $J=(f)$. Suppose that $\dim S/J=1$ and $f_1,\ldots,f_n$ is a regular sequence in $S$. Let  $\mathbf{m}=(x_0,\ldots,x_n)$ and $I=(J:\mathbf{m}^{\infty})$  be the saturation of $J$. Then $I_p=J_p$ for $p>>0$.

Let $H_{S/J}(t)$,  $H_{S/I}(t)$ be the Hilbert  series of $S/J$, respectively $S/I$. By \cite[Proposition 4.1.8]{BH} we have   $H_{S/J}(t)=G/(1-t)$,  $H_{S/I}(t)=F/(1-t)$ for some polynomials $G,F\in {\ZZ}[t]$ with $G(1)\not=0$, $F(1)\not=0$.

\begin{Lemma} \label{l} The following statements hold.
\begin{enumerate}
\item{} $G(1)=\dim_K(S/J)_p=\dim_K(S/I)_p=F(1)$ for $p>>0$;
\item{} $S/J$, $S/I$ have the same multiplicity $G(1)$;
\item{} $\dim_K(S/J)_p=\dim_K(S/I)_p$ for $p\geq u=\max\{\deg G,\deg F\}$, in particular $I_p=J_p$ for $p\geq u$;
\item{} $\deg G$ (resp. $\deg F$) is the minimum integer $p$ for which $\dim_K(S/J)_j=G(1)$ (resp. $\dim_K(S/I)_j=F(1)$) for all $j\geq p$.
\end{enumerate}
\end{Lemma}
\begin{proof}
We have $\dim_K(S/J)_p=G(1)$ for $p\geq \deg G$ and $\dim_K(S/I)_p=F(1)$ for $p\geq \deg F$. (1) holds because  $I_p=J_p$ for $p>>0$.  Also (3) holds from above. (2) follows from (1) by \cite[Proposition 4.1.9]{BH} and (4) holds since $\dim_K(S/J)_{\deg G-1}=G(1)-LC(G)\neq G(1)$, where $LC(G)$ is the leading coefficient of $G$.
\end{proof} Let $\Gamma\subset {\PP}^n$ be the zero dimensional complete intersection scheme with coordinate ring $S/(f_1,\ldots,f_n)$. Changing the coordinates we may suppose that $\Gamma\subset U_0={\PP}^n\setminus \{x_0=0\}$, where $\Gamma$ is given by $g_i=f_i(1,x_1,\ldots,x_n)$, $i\in [n]$. All local rings of $\Gamma$ are Artinian Gorenstein rings and we may apply the Cayley-Bacharach theory \cite{EGH}.
The support of $\Gamma$ consists of some points $p_1,\ldots,p_r$. A part of these points, let us say $p_1,\ldots,p_q$, $q\leq r$ are also in the support of $V(I)$. Let $\Gamma''$ be the closed subscheme of $\Gamma$ with the support $p_1,\ldots,p_q$ and the corresponding local ring at a point $p$ defined by the principal ideal generated by $g_0=f_0(1,x_1,\ldots x_n)$ in $A=K[x_1,\ldots,x_n]/(g_1,\ldots,g_n)$. In fact, $\Gamma''$ corresponds to the saturation $I$ of $J$.

Let $\Gamma'$ be the closed subscheme of $\Gamma$ with the support  $p_{q+1}\ldots,p_r$ and the corresponding local ring at a point $p$ defined by the ideal $\Ann_Ag_0$.  Then the subschemes $\Gamma'$,  $\Gamma''$ are residual to one another in the terminology from \cite{EGH}. In the same terminology, the {\em failure of $\Gamma''$ to impose independent conditions of hypersurfaces of degree $k$} means the multiplicity of $S/I$ minus $\dim_K (S/I)_k$, that is $F(1)-\dim_K(S/I)_k$.

\begin{Proposition} \label{p} Let $s=\sum_{i=1}^n d_i-n-1$ and $0\leq k\leq s$.  Then $\dim_K (((f_1,\ldots,f_n):f_0)/(f_1,\ldots,f_n))_{s-k}$ is the failure of $\Gamma''$ to impose independent conditions of hypersurfaces of degree $k$, that is  $F(1)-\dim_K(S/I)_k$.
\end{Proposition}
\begin{proof} Let $W'_e$ (resp. $W_e$) be the linear space of all $h\in S_e$ such that $h(\Gamma')=0$  (resp. $h(\Gamma)=0$). Note that
 $W'_e=(((f_1,\ldots,f_n):f_0))_e$, $W_e=((f_1,\ldots,f_n))_e$.
  By \cite[Theorem CB7]{EGH} the dimension of the $W'_e/W_e$ is  the failure of $\Gamma''$ to impose independent conditions of hypersurfaces of degree $k$, which is enough as we have seen.
\end{proof}

\begin{Lemma} \label{big}  $\deg F\leq s+1$ and $\deg G\leq s+1+d_0$.
\end{Lemma}
\begin{proof}
Set $W=(f_1,\ldots,f_n)\subset S$, $B=S/W$. We may suppose that $x_0$ is regular on $B$, $S/I$. We have $H_B(t)=E/(1-t)$, where $E=\Pi_{i=1}^n(1+t+\ldots t^{d_i-1})$ and it follows $\deg E=s+1$ and $H_{B/x_0B}(t)=E$,  $H_{S/(I,x_0)}(t)=F$ because $x_0$ is regular on  $S/I$, $B$. Since  $\dim_K (S/(I,x_0))_j\leq \dim_K B/(x_0B)_j$ for all $j$ we see that  $\deg F\leq \deg E=s+1$.

For the second inequality we consider the exact sequence
$$0\to ((W:_Sf_0)/W)(-d_0)\to B(-d_0)\xrightarrow{f_0} B\to S/J\to 0.$$
The Hilbert series $H_{(W:_Sf_0)/W}$ has the form $L/(1-t)$ for some $L\in {\ZZ}[t]$. As $x_0$ is regular on $B$ and in particular on the image of the above middle map we get the injection $ ((W:_Sf_0)/(W,x_0(W:_Sf_0))(-d_0)\to (B/(x_0))(-d_0)$. Thus \\
$ \dim_K((W:_Sf_0)/(W,x_0(W:_Sf_0))_j\leq \dim_K(B/(x_0))_j$ for all $j$. It follows that \\
$\deg L\leq \deg E= s+1$. If $j\geq s+1+d_0$ the above exact sequence
gives
$$\dim_K(S/J)_j=\dim_KB_j-\dim_KB_{j-d_0}+\dim_K((W:_Sf_0)/W)_{j-d_0}=E(1)-E(1)+L(1)$$
and so $\deg G\leq s+1+d_0$ and $G(1)=L(1)$.
\end{proof}

\begin{Theorem} \label{main}  With the above notation, we have
$$G=(1-t^{d_0})\Pi_{j=1}^n(1+t+\ldots+t^{d_j-1})+t^{\sum_{i=0}^n d_i-n}\ F(1/t).$$
Moreover, if $S/I$ is Gorenstein, then
$$G=(1-t^{d_0})\Pi_{j=1}^n(1+t+\ldots+t^{d_j-1})+t^{\sum_{i=0}^n d_i-n-\deg F}\ F.$$
\end{Theorem}
\begin{proof} By Proposition \ref{p} we have $\dim_K (W'/W)_{j}=F(1)-\dim_K (S/I)_{s-j}$, $0\leq j\leq s$, where  $W'=(W:_Sf_0)$. Let $F=\sum_{i=0}^{\deg F}w_it^i$, $w_i\in {\ZZ}$. We have $F'= t^{\deg F} F(1/t)=\sum_{i=0}^{\deg F}w'_it^i$, where $w'_i=w_{\deg F-i}$. Note that $F'(1)=F(1)$ and $\deg F'=\deg F$. If $S/I$ is Gorenstein we have $w_k=w_{\deg F-k}$ and so $F'=F$. If  $0\leq j\leq \deg F$
then $$\dim_K (S/I)_j=\sum_{k=0}^j w_k=F(1)-\sum_{k>j}^{\deg F} w_k=F'(1)-\sum_{k>j}^{\deg F} w'_{\deg F-k}=$$
$$F'(1)-\sum_{e=0}^{\deg F-j-1} w'_e.$$ In fact the above equality follows for all $j$ and we get  $$H_{W'/W}(t)=t^{s-\deg F+1}F'/(1-t)$$ for all $0\leq j\leq s$.

From  the exact sequence
$$0\to S/W'(-d_0)\xrightarrow{f_0} S/W\to S/J\to 0$$
we get $H_{S/J}(t)= H_{S/W}(t)-t^{d_0}H_{S/W'}(t)$.
But
$$\dim(S/W')_j=\dim(S/W)_j-\dim(W'/W)_j=\dim(S/W)_j-\dim (S/I)_{\deg F-s+j-1}$$
for $0\leq j\leq s$. We have $H_{S/W}(t)=[\Pi_{i=1}^n (1+\ldots +t^{d_i-1})]/(1-t)$ since $\{f_1,\ldots,f_n\}$ is regular. Then
$$H_{S/W'}(t)\equiv [(\Pi_{i=1}^n (1+\ldots +t^{d_i-1}))/(1-t)]-[t^{s+1-\deg F}F'/(1-t)]$$
 modulo $t^{s+1}$
  since  $\deg F\leq s+1$ by Lemma \ref{big}. It follows that
  $$H_{S/J}(t)\equiv [((1-t^{d_0})(\Pi_{i=1}^n (1+\ldots +t^{d_i-1})))/(1-t)] +[t^{s+d_0-\deg F+1}F'/(1-t)]$$
modulo $t^{s+d_0+1}$.
On the other hand, for $j\geq s+d_0+1$ we have   $\dim_K (S/J)_j=G(1)$ because $\deg G\leq s+1+d_0$ by Lemma \ref{big}. As $\dim_K (S/I)_k=F(1)=G(1)$ for $k\geq \deg F$ we see that the above congruence is in fact an equality, which is enough.
 \end{proof}

\begin{Proposition} \label{inj}  The map $(S/W)_j\xrightarrow{f_0} (S/W)_{j+d_0}$ is injective for $j\leq v=s-\deg F$ but is not for $j=v+1$.
\end{Proposition}
\begin{proof} As in the beginning of the proof of the above theorem, we have the  equality $\dim_K(W'/W)_j=F(1)-\dim_K (S/I)_{s-j}$
where $W'=(W:_Sf_0)$. Clearly, $\dim_K(S/I)_{s-j}=F(1) $ for $j\leq v$ but not for $j=v+1$, which is enough.
\end{proof}

Consider the graded $S-$submodule $AR({\bf f}) \subset S^{n+1}$ of {\it all relations} involving the $f_j$'s, namely the sums of homogeneous elements of degree $m$
$$a=(a_0,...,a_n) \in AR({\bf f})_m,$$
where, by definition, each $a_j$ a homogeneous polynomial in $S$ of degree $m-d_j$ for some integer $m$ and
\begin{equation}
\label{rel}
 (R_m):  \   \    \  \   a_0f_0+a_1f_1+...+a_nf_n=0.
\end{equation}
Inside $AR({\bf f})$ there is the $S-$submodule of {\it Koszul relations} $KR({\bf f})$, called also the
submodule of {\it trivial relations}, spanned by the relations $t_{ij} \in AR({\bf f})_{d_id_j-1}$ for
$0 \leq i <j \leq n$, where $t_{ij}$ has the $i$-th coordinate equal to $f_j$, the $j$-th coordinate equal to $-f_i$ and the other coordinates zero.

The quotient module $ER({\bf f})=AR({\bf f})/KR({\bf f})$ may be called the module of {\it essential relations} of the sequence ${\bf f}$, or non trivial relations, since it tells us which are the relations which we should add to the Koszul relations in order to get all the relations, or syzygies, involving the $f_j$'s.

The following result is obvious, the last claim coming from Proposition \ref{inj} .

\begin{Corollary} \label{ident}
The mapping $ER({\bf f})_{j+d_0 }  \to \ker \{(S/W)_j\xrightarrow{f_0} (S/W)_{j+d_0}\} $ given by
$$(a_0,...,a_n) \mapsto [a_0]$$
induces an isomorphism of graded $S$-modules $ER({\bf f})_{ }  \to \ker \{(S/W)\xrightarrow{f_0} (S/W)(d_0)\} $. In particular, the minimal degree of a relation \eqref{rel} is $m=\sum_{i=0}^n d_i-n-\deg F.$

\end{Corollary}

To get precise numerical results, consider the case when the ideal $I$ is a complete intersection of type $(a_1,...,a_n)$, i.e. $I=(g_1,...,g_n)$ for some homogeneous polynomials of degree $\deg g_j=a_j \geq 1$. We can suppose that the integers $a_j$ satisfy
$$a_1 \leq a_2 \leq ... \leq a_n.$$
Then we have the following result.

\begin{Theorem}\label{comput} Suppose $d_1 \leq d_2 \leq ... \leq d_n$. Then the following statements hold.
\begin{enumerate}
\item{} $F=\Pi_{j=1}^n(1+\ldots +t^{a_j-1})$, in particular we have the following formulas \\
$F(1)=a_1a_2\cdots a_n$ and $\deg F=\sum_{j=1}^n a_j-n.$
\item{} $a_j\leq d_j$ for $j=1,...,n$ and $d_0 \geq a_1$.
\item{} If $d_0 >a_1$  or $d_0=a_1=a_2$, then
$$\deg G=\sum_{i=0}^n d_i-n-a_1 \geq \deg F.$$
\end{enumerate}
\end{Theorem}
\begin{proof} The first claim is clear, since $g_1,...,g_n$ form a regular sequence.
For any $j$, we have $f_j \in I$, hence $d_j \geq a_1$. In particular this holds for $j=0$ and $j=1$.
Assume that $d_2<a_2$. This would imply the inclusion of ideals $(f_1,f_2) \subset (g_1)$, a contradiction because $f_1,\ldots,f_n$ is regular. One continues in this way to prove the second claim.

For the proof of the third claim, we write  the polynomial $G$ using the equality in Theorem \ref{main} in the following form
$$\Pi_{j=1}^n(1+t+\ldots+t^{d_j-1})-t^{d_0}[\Pi_{j=1}^n(1+t+\ldots+t^{d_j-1})-\Pi_{j=1}^n(t^{d_j-a_j}+\ldots+t^{d_j-1})].$$
The first product yields powers of $t$ up to $a=\sum_{i=1}^n d_i-n$ while the second term will produce powers of $t$ up to at most $b=\sum_{i=0}^n d_i-n-a_1$ due to obvious simplifications in the difference in the bracket. If $d_0>a_1$, then $b>a$ and the claim is established. If $d_0=a_1=a_2$, then $a=b$,  the coefficient of $t^a$ coming from the first product is $1$ and the coefficient of $t^a$ coming from the difference is $< -1$.

  Finally note that
 $\deg G-\deg F= (d_0-a_1)+ \sum_{i=1}^n (d_i-a_i)\geq 0$ by (1).
 \end{proof}

\section{Lefschetz properties of dimension one  almost complete intersections }

Let $A$ be a standard graded  $K$-algebra and $h\in A$ a homogeneous form of degree $k$. The element $h$ is called a {\em Lefschetz element} if for all integers $i$ the $K$-linear map $h:A_i\to A_{i+k}$ induced by the multiplication with $h$ has maximal rank (usually we ask for $A$ to be Artinian in this definition). We need the following theorem (see \cite[Theorem 2.3]{HMNW}, or \cite[Corollary 2.4]{BK}) :

\begin{Theorem}(Harima-Migliore-Nagel-Watanabe) Let $n=2$ and $A=S/(f_0,f_1,f_2)$ be a graded complete intersection, that is $\{f_0,f_1,f_2\}$ form a system of parameters of $S$.  Then there exists a linear form $l$ of $S$ which induces a Lefschetz element for $A$.
\end{Theorem}

\begin{Proposition} In the notation and assumptions of the above theorem, for any   $0\leq i\leq 2$ there is  a change of coordinates such that $f_i$ induces a Lefschetz element for $S/(x_0,f_0,\ldots,f_{i-1},f_{i+1},\ldots f_2)$.
\end{Proposition}
\begin{proof} Let $l\in S_1$ be the form given by the above theorem for $A$. We may suppose that $l=\sum_{i=0}^n b_ix_i$, $b_i\in K$, $b_0\not =0$. Taking the $K$-automorphism $\varphi$ of $S$ given by $x_0\to l$ and $x_i\to x_i$ for $1\leq i\leq n$ we see that $x_0$ induces a Lefschetz element for $S/{\varphi}^{-1}(f_0,f_1,f_2)$, that is we may suppose after a change of coordinates that $x_0$ induces a Lefschetz element for $S/(f_0,f_1,f_2)$. Then by \cite[Lemma 1.1]{HP} (see also \cite[Lemma 3.1]{P}), we see that  $f_0$ induces a Lefschetz element for $S/(x_0,f_1,f_2)$.
\end{proof}

Next we study a  similar property for almost complete intersections.  Let $ S=K[x_0,\ldots,x_n]$ and  $f_i$, $0\leq i\leq n$ be some homogeneous polynomials of $S$ with degree $d_i$. Set $J=(f_0,\ldots,f_n)$. Suppose that $f_1,\ldots,f_n$ is regular in $S$. We may choose a linear form  $l$  regular on $B=S/(f_1,\ldots,f_n)$. Set $J'=(J,l)$. Then $S/J'$ is Artinian and let $L=H_{S/J'}(t)\in {\ZZ}[t]$. Since $\dim_K(S/J')_j\leq \dim_K(S/(l,f_1,\ldots,f_n))_j=0$ for $j>q=\sum_{i=1}^n (d_i-1)$ we see that $\deg L\leq q$ (in our previous notation from Proposition \ref{p} and Lemma \ref{big},   $q=s+1$).

 \begin{Proposition}  \label{p1}  The following statements are equivalent for an integer $p\in {\NN}$.
\begin{enumerate}
\item{}
 the map $(B/(l))_{p-d_0}\xrightarrow{f_0} (B/(l))_p$ is surjective,
 \item{}
 the map $(B/(l))_{q-p}\xrightarrow{f_0} (B/(l))_{q-p+d_0}$ is injective.
 \end{enumerate}
 In particular, both statements hold when $p>\deg L$.
 \end{Proposition}
 \begin{proof}
Since $B/(l)$ is an Artinian Gorenstein ring it is enough  to see that  $\nu_p:(B/(l))_{q-p}\xrightarrow{f_0} (B/(l))_{q-p+d_0}$ is given by duality by $\nu_{p-d_0}$.
More precisely, the canonical isomorphisms $\varphi_p:(B/(l))_{q-p}\to \Hom_K((B/(l))_p,(B/(l))_q)$, $0\leq p\leq q$ which maps $z\in (B/(l))_{q-p}$ in the map $(B/(l))_p\xrightarrow{z} (B/(l))_q$ make commutative the following diagram
\begin{displaymath}
\xymatrixcolsep{10pc}\xymatrix{
(B/(l))_{q-p} \ar[d] \ar[r]^{\nu_p} & (B/(l))_{q-p+d_0} \ar[d] \\
\Hom_K((B/(l))_{p},(B/(l))_{q}) \ar[r]^{\Hom_K(\nu_{p-d_0},(B/(l))_{q})}  & \Hom_K((B/(l))_{p-d_0},(B/(l))_{q})
}
\end{displaymath}
the vertical maps being $\varphi_p$ and $\varphi_{p-d_0}$.
\end{proof}

It remains to study the integer $\deg L$.

\begin{Proposition}\label{p2} Let  ${\tilde W}=(l,W)$ and suppose that the following assumption holds
$$(A)   \  \  \text {   there exists an integer } k, \   0\leq k<q  \text{ such that } (f_0)_{d_0+k}\cap {\tilde W}_{d_0+k}=f_0{\tilde W}_k.$$
Then $\deg L\leq q-k-1$.
\end{Proposition}
\begin{proof}
The image of the map $\alpha_k:(B/(l))_k\xrightarrow{f_0} (B/(l))_{k+d_0}$ is isomorphic with  $((f_0)/(f_0)\cap {\tilde W})_{d_0+k}\cong ( S/{\tilde W})_k=(B/(l))_k$ using our hypothesis. Therefore $\alpha_k $ is injective and by Proposition \ref{p1} we get $\alpha_{q-k-d_0}$ surjective, that is $(S/(J,l))_{q-k}=(B/(f_0,l))_{q-k}=0$ and so $\deg L<q-k$.
\end{proof}

\begin{Corollary} If $f_0\not \in (l,f_1,\ldots,f_n)$ then $\deg L<q$.
\end{Corollary}
For the proof apply the above proposition for $k=0$.

\begin{Theorem} \label{m1} The following statements hold.
\begin{enumerate}
\item{} For any integer $p$, the morphism
 $(B/(l))_{p-d_0}\xrightarrow{f_0} (B/(l))_p$ is surjective if and only if the morphism $(S/J)_{p-1}\xrightarrow{l} (S/J)_p$ is surjective. This surjectivity holds for all $p\geq q-k$ if the assumption $(A)$ above is true.
 \item{} For any integer $r$, if
 the morphism $(B/(l))_{r}\xrightarrow{f_0} (B/(l))_{r+d_0}$ is injective, then the morphism $(S/J)_{r+d_0-1}\xrightarrow{l} (S/J)_{r+d_0}$ is also injective. This injectivity holds for all $r\leq k$ or $r>q$ if the assumption $(A)$ above is true.
 \end{enumerate}
 \end{Theorem}
\begin{proof} The first claim in (1) follows from the fact that both surjectivities are equivalent to
the equality $J'_p=S_p$.
Apply then Proposition \ref{p1} and the proof of Proposition \ref{p2}. If $r>q$ then $(B/(l))_r=0$ and the second part of (2) is trivial. To prove the first part of (2), it is enough to show that if $(B/(l))_{e-d_0}\xrightarrow{f_0} (B/(l))_e$ is injective for some $e\in {\NN}$ then  $(S/J)_{e-1}\xrightarrow{l} (S/J)_e$ is also injective. This follows easily from Koszul homology. Here we will give a simple direct proof.

Let $\omega\in S_{e-1}$ be such that $l\omega\in (J)_e$, that is $l\omega\equiv f_0h_0$ modulo  $W$
for some $h_0\in S_{e-d_0}$. Then $h_0\in {\tilde W}_{e-d_0}$ by injectivity assumption, let us say $h_0\equiv l\gamma_0$ modulo $W$ for some $\gamma_0\in S_{e-d_0-1}$. It follows that $l (\omega-f_0\gamma_0)\in W$ and so $\omega\equiv f_0\gamma_0$ modulo $W$ because $l $ is regular on $B$. Thus $\omega\in J_{e-1}$, which is enough.
\end{proof}

\begin{Remark} {\em The converse implication in (2) does not hold in general, see Example \ref{Ekey}. The above theorem works for any linear form $l$ regular on $B$ with $f_0\not\in {\tilde W}$. But it is important to find such $l$ with $(f_0)_{d_0+k}\cap {\tilde W}_{d_0+k}=f_0{\tilde W}_k$ for high $k$. Clearly the highest such $k$ is smaller then $v$ given by Proposition \ref{inj} but in general is strictly smaller   as shows the following example.}
\end{Remark}

\begin{Example} {\em Let $n=1$, $f_0=x_0x_1$, $f_1=x_1^{d_1}$ for some $d_1>2$. We may take $l=x_0$ because $x_0$ is regular on $B=S/(f_1)$ but then $f_0\in {\tilde W}=(l,f_1)$ and there exist no $k$ as above. However for $l=x_0-x_1$ we see that $k=d_1-3$ works. Indeed, if $f_0g\in {\tilde W}$ for some nonzero homogeneous polynomial $g$ of degree $k$ then $f_0g$ has degree $d_1-1$ and induces zero in $S/{\tilde W}\cong K[x_1]/(x_1^{d_1})$, which is false. It is easy to see that $k>d_1-3$ does not work. Applying Proposition we have $v=s-\deg F==d_1-2$ since $F=1$ and so $k=v$ does not work.}
\end{Example}

The idea of this example gives easily the following corollary.
\begin{Corollary} \label{n=1} Let $n=1$ and $f_i\in S_{d_i}$, $i=0,1$ be such that $\height (f_0,f_1)=1$ and $d_0<d_1$. Then there exists $l\in S_1$ regular on $S/(f_1)$ and such that
\begin{enumerate}
 \item{}
 the maps $(B/(l))_{r}\xrightarrow{f_0} (B/(l))_{r+d_0}$ and $(S/J)_{r+d_0-1}\xrightarrow{l} (S/J)_{r+d_0}$ are injective for all $r\leq d_1-d_0-1$ or $r>d_1-1$,

\item{}
 the maps $(B/(l))_{p-d_0}\xrightarrow{f_0} (B/(l))_p$ and $(S/J)_{p-1}\xrightarrow{l} (S/J)_p$ are surjective for all $p\geq d_0$.
\end{enumerate}
In particular, $l$ is a Lefschetz element for $S/(f_0,f_1)$.
\end{Corollary}
\begin{proof}
As in the above example we see that one can find $l$  for which  $k=d_1-d_0-1$ works and we may apply Theorem \ref{m1}. Then note that  $(S/J)_{p-1}\xrightarrow{l} (S/J)_p$ is  injective for $p\leq d_1-1$ and surjective for $p\geq d_0$. As $d_0\leq d_1-1$ we see that  $l$ is a Lefschetz element for $S/(f_0,f_1)$.
\end{proof}

\begin{Remark} {\em In particular \cite[Proposition 4.4]{HMNW} says that if $A$ is a graded Artinian factor of $K[x_0,x_1]$, then it has a Lefschetz element in $A_1$. This result holds somehow also in the case $\dim A=1$ as shows our above corollary.}
\end{Remark}

\section{The case of projective hypersurfaces with isolated singularities}

 Let $V\subset {\PP}^n$ be a projective hypersurface defined by a homogeneous polynomial $f\in S$ of degree $d\geq 2$ and assume $n\geq 2$. Let $f_0,\ldots,f_n$ be its partial derivatives, $J_f=(f_0,\ldots,f_n)$ be the Jacobian ideal of $f$ and  $M(f)=S/J_f$ be the Milnor algebra of $f$. Let $\mathbf{m}=(x_0,\ldots,x_n)$ and $I=(J_f:\mathbf{m}^{\infty})$ be the saturation of $J_f$.

Assume that $V$ has only isolated singularities, i.e.  $\dim S/J_f=1$. Then, maybe after a coordinate change, we can assume that $\{f_1,\ldots f_n\}$ is a regular sequence.

Recall that the {\em stability threshold} $st(V)$ of $V$ was defined in \cite{DSt} as $$st(V)=\min\{p:\dim_K M(f)_k=\tau(V) \mbox{ for\ \ all}\ \ k\geq p\},$$
where $\tau(V)$ is the total Tjurina number of $V$.
Moreover, the {\it coincidence threshold} $ct(V)$ was defined as
$$ct(V)=\max \{q~~:~~\dim_K M(f)_k=\dim M(f_s)_k \text{ for all } k \leq q\},$$
with $f_s$  a homogeneous polynomial in $S$ of degree $d$ such that $V_s:f_s=0$ is a smooth hypersurface in $\PP^n$. Finally, the {\it minimal degree of a nontrivial relation} $mdr(V)$ is defined as
$$mdr(V)=\min \{q~~:~~ ER({\bf f})_{q-d+1}\ne 0\}.$$
It is known that one has
\begin{equation} 
\label{REL}
ct(V)=mdr(V)+d-2,
\end{equation} 
see \cite{DSt}, formula (1.3).

As in the first section note that $H_{M(f)}(t)$ has the form $G/(1-t)$, where $G\in {\ZZ}[t]$ with $G(1)\not=0$, $G(1) $ being the multiplicity of $M(f)$ (see \cite[Proposition 4.1.9]{BH}). We have the following result.

\begin{Lemma} \label{L} With the above notation,  the following hold.

\begin{enumerate}
\item{} $\tau(V)=G(1)$ is the multiplicity of $M(f)$ and $\deg G=st(V)$ is the stability threshold of $V$.

\item{} $\tau(V)=F(1)$ is the multiplicity of $S/I$ and $\deg F= (n+1)(d-2)-ct(V)$.

\end{enumerate}
\end{Lemma}
\begin{proof} We have $\dim_KM(f)_k=G(1)$ for all $k\geq \deg G$ and so $\tau(V)=G(1)$, $st(V)\leq \deg G$. In fact $st(V)=\deg(G)$ since $\dim_KM(f)_{\deg G-1}=G(1)-LC(G)$, where $LC(G)$ is the leading coefficient of $G$. The first part of the second claim follows from Lemma \ref{l}, while the second part is a consequence of \cite[Proposition 2]{D}.
\end{proof}

\begin{Remark} \label{rk1}{\em In the notation from Lemma \ref{l}, one has $s=n(d-1)-n-1$, $s+1+d_0=(n+1)(d-2)+1$ and Lemma \ref{l}, (4) and Lemma \ref{big} imply $st(V) \leq (n+1)(d-2)+1$, which is sharp for $V$ smooth, but far from sharp when $V$ has isolated singularities, see \cite{DSt}.}
\end{Remark}

 \begin{Remark} \label{rk2}{\em The last claim in Corollary \ref{ident} tells us that $\deg F =(n+1)(d-2)-(m-1)$, which is clearly equivalent to the last claim in Proposition 2 in \cite{D}. To see this, use the formula \eqref{REL} and note the twist in the definition of the degree of a relation in the first and the third section.  }
 \end{Remark}

Let $l$ be a  linear form which is   regular  $B=S/(f_1,\ldots,f_n)$. As in the previous section set $J'=(J_f,l)$, $L=H_{S/J'}(t)\in {\ZZ}[t]$. We have $\deg L\leq q=n(d-2)$.
The following theorem follows from Theorem \ref{m1}.
\begin{Theorem} \label{m2} Suppose that there exists $k$,   $0\leq k<q=n(d-2)$ such that
 $(f_0)_{d+k-1}\cap {\tilde W}_{d+k-1}=f_0{\tilde W}_k$, where ${\tilde W}=(l,f_1,\ldots,f_n)$. Then the following statements holds:
\begin{enumerate}
\item{}
 the maps $(B/(l))_{p-d+1}\xrightarrow{f_0} (B/(l))_p$ and $M(f)_{p-1}\xrightarrow{l} M(f)_p$ are surjective for all $p\geq q-k$,
 \item{}
 the maps $(B/(l))_{r}\xrightarrow{f_0} (B/(l))_{r+d-1}$ and $M(f)_{r+d-2}\xrightarrow{l} M(f)_{r+d-1}$ are injective for all $r\leq k$ or $r>q$.
 \end{enumerate}
 \end{Theorem}
 \begin{Remark} {\em After a change of coordinates we may suppose that $x_0=l$
  in the above theorem, using the next elementary lemma.}
  \end{Remark}
\begin{Lemma} \label{ele}
Let $l=\sum_{i=0}^n b_ix_i$, $b_i\in K$, $b_0\not =0$ be a linear form of $S$ and ${\varphi }$ be the $K$-automorphism of $S$ given by $x_0\to l$ and $x_i\to x_i$ for $1\leq i\leq n$. Then one has
 $\varphi(J_f)=J_{{\varphi}(f)}$.
\end{Lemma}
\begin{proof} We have $\partial{\varphi}(f)/\partial x_i=a_i{\varphi}(\partial f/\partial x_o)+{\varphi}(\partial f/\partial x_i)$ for $1\leq i\leq r$
and \\
$\partial{\varphi}(f)/\partial x_0=a_0{\varphi}(\partial f/\partial x_o)$. Thus $\varphi(J_f)=J_{{\varphi}(f)}$ since $a_0\not=0$.
\end{proof}

\begin{Example}\label{Ekey} {\em Assume that $J_f=I$. In the case $n=2$ this corresponds exactly to the case when $V$ is a free divisor, see \cite[Remark 4.7]{DS14}. As concrete examples one can take $f=(x_0^2-x_1^2)(x_1^2-x_2^2)(x_0^2-x_2^2)$ or $f=(x_0^3-x_1^3)(x_1^3-x_2^3)(x_0^3-x_2^3)$. For many examples of (irreducible) free divisors see \cite{BC} and \cite{ST}.

When $J_f=I$ it follows from \cite[Proposition 2]{D} that $M(f)_{p-1}\xrightarrow{l} M(f)_p$ is surjective if and only if $p-1 \geq st(V)$. This follows from the fact that $M(f)=S/I$ can be regarded in this case as a free graded $K[l]$-module, see \cite[Remark 1.7]{DS}. This also implies that $M(f)_{p-1}\xrightarrow{l} M(f)_p$ is injective for any $p$. In view of Proposition \ref{p1}, this shows that the implication in Theorem \ref{m1} (2) is not an equivalence.}

\end{Example}

Suppose moreover from now on that  the singular locus  $Y$ of $V$ is a $0$-dimensional complete intersection. Then $I=(g_1,\ldots,g_n)$ for some regular sequence $g_1,\ldots,g_n$ with $g_i\in S_{a_i}$ for some $a_i\in {\NN}$ and $Y=V(I)$. We have  $(J_f)_p=I_p$ for $p>>0$.

\begin{Lemma} \label{L2} With the above notation,  the following are equivalent.

\begin{enumerate}
\item{} $a_1=...=a_n=d-1.$
\item{} $f_0$ belongs to the $K$-vector space spanned
 by $f_1,...,f_n$.
\item{} $f_0$ belongs to the ideal spanned by $f_1,...,f_n$.
\item{} After a linear coordinate change, $f$ is a polynomial in $x_1,...,x_n$ only.
\end{enumerate}
In other words, $V$ is the cone over a smooth hypersurface in ${\PP}^{n-1}.$
\end{Lemma}
\begin{proof} This proof is obvious and we leave it for the reader.
\end{proof}

Now we list some consequences of our results in the first section.

 \begin{Proposition} \label{CIcase}(Choudary-Dimca, \cite[Proposition 13]{CD}, Dimca \cite[Proposition 4]{D}) Then
the Hilbert Poincare series of $M(t)$ is given by $$H_{M(f)}(t)=(1/(1-t)^{n+1})[(1-t^{d-1})^{n+1}+t^{(n+1)(d-1)-\sum_{i=1}^n a_i}\ \ \Pi_{j=1}^n(1-t^{a_j})].$$
 \end{Proposition}
The proof follows directly from Theorem \ref{main} because in this case
$$F=\Pi_{j=1}^n(1+\ldots +t^{a_j-1}).$$
Using now in addition Theorem \ref{comput}, we get the following.

\begin{Corollary}\label{corct} (compare to Dimca \cite[Corollary 5]{D})
With the above notation and assumptions, we have \\
$\tau(V)=a_1\cdots a_n$,  $ct(V)=T-\sum a_i +n$ and  $st(V)=(n+1)(d-2)+1-\min_ia_i$. \\
Moreover, in this case one has $\deg G\geq \deg F$.
\end{Corollary}

By Lemma \ref{l} and Theorem \ref{comput} (see also \cite[Corollaries 1, 5]{D}) we get the following.
 \begin{Corollary} \label{l1} $(J_f)_p=I_p$ if $p\geq \deg G=(n+1)(d-1)-n-\min_ia_i$.
 \end{Corollary}

The following example considers plane curves with at most three nodes as singularities.

\begin{Example} (compare to \cite[Example 4.3]{DSt}) {\em Let $V$ be a plane curve. If $V$ has a single node, it is clear that the singular locus $Y$ of $C$ is a complete intersection of type $(a_1,a_2)=(1,1)$ and hence $st(V)=3d-6$. If $V$ has two nodes, then $Y$ is a complete intersection of type $(a_1,a_2)=(1,2)$ and again $st(V)=3d-6$. Suppose $C$ has three nodes. Then $Y$ is a complete intersection if and only if the nodes a collinear, and in this case has  type $(a_1,a_2)=(1,3)$. Once more one has $st(V)=3d-6$ in this latter case.}
\end{Example}

The following example discusses a large family of hypersurfaces having the singular locus $Y$ a 0-dimensional complete intersection, and hence for which the formulas given in Corollary \ref{corct} hold.

\begin{Example} \label{E1} (compare to \cite[Theorem 6.4.14, p. 211]{book2}) {\em
 Let $b_1,...,b_n \geq 1$ and $c_1,...,c_n \geq 2$ be two families of integers such that $b_jc_j=d$ for all $j=1,...,n$, and some fixed integer $d$. Consider the polynomials $g_j=x_0^{b_j}+x_j^{b_j}$ and define the homogeneous polynomial $f \in S_d$ by the sum
$f=g_1^{c_1}+ \cdots +g_n^{c_n}$. It is easy to see that the hypersurface $V:f=0$ has exactly $b=b_1\cdots b_n$ singularities, the common zeros of the polynomials $g_j$ for $j=1,...,n$, and at each singular point $p$ the germ of $V$  is given by a local equation of the type $u_1^{c_1}+ \cdots +u_n^{c_n}=0$, with $(u_1,...,u_n)$ a local (analytic) system of coordinates centered at $p$.
It follows from this description that the ideal $I$ coincides with the ideal spanned by $g_j^{c_j-1}$
for $j=1,...,n$, and hence $Y$ is a complete intersection of type $(a_1,...,a_n)$ with $a_j=b_j(c_j-1)=d-b_j$.
It follows that
$$\tau(V)=(d-b_1)\cdots (d-b_n), \ ct(V)=\sum_j b_j -n +d-2$$
and
$$  st(V)=(n+1)(d-2)+1-\min_i (d-b_i).$$
It follows that 
\begin{equation} 
\label{REL2}
mdr(V)=\sum_j b_j -n=d(\sum_j \frac{1}{c_j})-n =d\alpha_V -n,
\end{equation} 
where $\alpha_V=\sum_j \frac{1}{c_j}$ is the Arnold exponent of the singularity $u_1^{c_1}+ \cdots +u_n^{c_n}=0$, see for instance \cite{AGV}.\\

To get nodal hypersurfaces of even degree, we can set $b_1=...=b_n=d_1$, $c_1=...=c_n=2$ and $d=2d_1$.
It follows that in this case we get $ct(V)=nd_1+d-n-2$ and \cite[Theorem 4.1]{DSt3} combined with Proposition \ref{CIcase} above implies that the corresponding Alexander polynomial is given by the formula
$$\Delta_V(t)=t+(-1)^{n+1}.$$
This is the only case (to the best of our knowledge) when such nontrivial Alexander polynomials can be completely determined for series of hypersurfaces of arbitrary dimension. In relation with the last formula, which implies that the Betti number $b_{n-1}(F)$ of the Milnor fiber given by  $F:f=1$ in ${\CC}^{n+1}$ satisfies  $b_{n-1}(F)=1$, one can see also \cite{vS} for more general, but less precise results.
}
\end{Example}

Finally, we would like to state the following two conjectures. The first one is made in view of Theorem \ref{comput} (3) and the evidence given in
\cite[Example 2]{D}.

\begin{Conjecture} \label{C1} For any projective degree $d$ hypersurface $V:f=0$ in $ {\PP}^n$
having only isolated singularities, one has
$$st(V)=\deg G\geq \deg F=(n+1)(d-2)-ct(V).$$
\end{Conjecture}

The second one is made in view of the formula \eqref{REL2} above and 
\cite[Theorem 2.1]{DS14} saying that the claim holds for $n=2$.

\begin{Conjecture} \label{C2} For any projective  degree $d$ hypersurface $V:f=0$ in $ {\PP}^n$
having only weighted homogeneous isolated singularities, one has
$$mdr(V) \geq d\alpha_V -n, $$
with $\alpha_V$  the minimum of the Arnold exponents of the singularities of $V$.
\end{Conjecture}

\section{A Lefschetz property in the case $n=2$}

In this section we assume $n=2$ and discuss Lefschetz type properties for the graded $S$-module $H^0_\mathbf{m}(S/J)=I/J$, the $0$-degree local cohomology of the corresponding  algebra $S/J$. 
Let $N$ be a  graded  $S$-module and $h\in S_1$ a homogeneous form of degree $1$. The element $h$ is called a {\em Lefschetz element for $N$} if for all integers $i$ the $K$-linear map $h:N_i\to N_{i+1}$ induced by the multiplication with $h$ has maximal rank. Our main result is the following.

\begin{Theorem}  \label{mainL} Let $J=(f_0,f_1,f_2)$ be a dimension one almost complete intersection and let  $N=H^0_\mathbf{m}(S/J)=I/J_f$ be the $0$-degree local cohomology of the corresponding  algebra $M=S/J$. Then there exists a Lefschetz element for $N$. More precisely,
for a generic linear form $l \in S_1$, the multiplication by $l$ induces injective morphisms $N_i \to N_{i+1}$ for $i<(d_0+d_1+d_2-3)/2$ and surjective morphisms 
$N_i \to N_{i+1}$ for $i \geq i_0=[(d_0+d_1+d_2-3)/2]$.
\end{Theorem}

\begin{proof} Consider the exact sequence of sheaves on $\PP^2$ given by
$$0 \to \mathcal K \to \mathcal O(-d_0)\oplus \mathcal O(-d_1) \oplus \mathcal O(-d_2) \to  \mathcal J \to 0,$$
where $\mathcal K$ is the syzygy bundle as in \cite{BK} and $ \mathcal J$ is the sheaf ideal in $\mathcal O= \mathcal O_{\PP^2}$ associated to the ideal $J$. For any integer $m \in \ZZ$, one has $H^0(\PP^2, \mathcal J(m))=I_m$, and hence this exact sequence yields an isomorphism
$$H^1(\PP^2,\mathcal K(m))=H^0 _\mathbf{m}(S/J)_{m}=I_m/J_m=N_m.$$

If the sheaf $\mathcal K$ is semistable, then the proof follows exactly the same path as the proof of \cite[Theorem 2.2 (1)]{BK}. See also \cite[Remark 2.3]{BK}.
And when the sheaf $\mathcal K$ is not semistable, then the proof follows the same approach as in the second part in the proof of \cite[Corollary 2.4]{BK}.

The last claim follows using semi-continuity properties of the rank of a linear mapping and the duality properties for $N$, see \cite[Theorem 3.2]{Se}.

\end{proof}

This implies a partial Lefschetz type property for the algebra $M=S/J$, namely we have the following.

\begin{Corollary} \label{corLA} For a generic linear form $l \in S_1$, the multiplication by $l$ induces injective morphisms $M_i \to M_{i+1}$ for $i<(d_0+d_1+d_2-3)/2$. The morphisms 
$M_i \to M_{i+1}$ for $i \geq i_0=[(d_0+d_1+d_2-3)/2]$ are surjective if and only if $i_0\geq \deg F.$

 \end{Corollary}

\begin{proof} This claim follows from Theorem \ref{mainL} and the fact that $M/N=S/I$ can be regarded  as a free graded $K[l]$-module, see \cite[Remark 1.7]{DS}. Indeed, the condition
$i_0 \geq \deg F$ is equivalent to $\dim M_{i_0}=\dim N_{i_0} + F(1)$. And the last equality says exactly that all the generators of the free graded $K[l]$-module $M/N$ occur in degree at most $i_0$ and hence $l: (M/N)_i \to (M/N)_{i+1}$ for $i \geq i_0$ is surjective. The surjectivity claim follows using the exact sequence
$ 0 \to N_m \to M_m \to (M/N)_m \to 0.$
\end{proof}

\bigskip

From now on  we consider  the case of Jacobian ideals of plane curves,  and note that the graded $S$-module $H^0_\mathbf{m}(M(f))=I/J_f$, the $0$-degree local cohomology of the corresponding Milnor algebra $M(f)$, occurs in a number of recent preprints, see
\cite{DS14}, \cite{DS},  \cite{Se}. The above results give the following.

\begin{Corollary}  \label{corL0} Let $C:f=0$ be a degree $d$ reduced plane curve and let  $N(f)=H^0_\mathbf{m}(M(f))=I/J_f$ be the $0$-degree local cohomology of the corresponding Milnor algebra $M(f)$. Then there exists a Lefschetz element for $N(f)$. More precisely,
for a generic linear form $l \in S_1$, the multiplication by $l$ induces injective morphisms $N(f)_i \to N(f)_{i+1}$ for integers $i<(3d-6)/2$ and surjective morphisms 
$N(f)_i \to N(f)_{i+1}$ for integers $i \geq i_0=[(3d-6)/2]$.
\end{Corollary}

This implies a partial Lefschetz type property for the Milnor algebra $M(f)$, namely we have the following.

\begin{Corollary} \label{corL} For a generic linear form $l \in S_1$, the multiplication by $l$ induces injective morphisms $M(f)_i \to M(f)_{i+1}$ for $i<(3d-6)/2$. The morphisms 
$M(f)_i \to M(f)_{i+1}$ for $i \geq i_0=[(3d-6)/2]$ are surjective if and only if $ct(C)\geq 3(d-2)-i_0.$

 \end{Corollary}

\begin{Remark} \label{rkL0} {\em The above condition $ct(C)\geq 3(d-2)-i_0$  is satisfied in many cases, e.g. for all nodal curves since they satisfy $ct(C) \geq 2d-4$ by  \cite[Theorem 1.2]{DSt}.  However, there are curves for which this property fails, e.g. the plane curve $C:x^py^q+z^d=0$ for $p>0$, $q>0$ and $p+q=d$. Then $ct(C)=mdr(C)+d-2=d-1< 3(d-2)-i_0.$}
  \end{Remark}
 
\begin{Remark} \label{rkL} {\em A consequence of Corollary \ref{corL0} and of the duality results for the module $N(f)$ obtained in  \cite{DS},  \cite{EM},  \cite{Se} is that we have the following relations involving the sequence of dimensions $n_k=\dim N_k$: $n_k=n_{3d-6-k}$ for any $k$ and
$$0 = n_0 \leq n_1 \leq ... \leq n_{i_0} \geq n_{i_0+1} \geq ... \geq n_{3d-6}= 0,$$
with $i_0=[(3d-6)/2]$. In other words, the sequence $(n_k)$ is unimodal, a property conjectured in  \cite[Remark 6]{D}.  A stronger property holds when $I$ is a complete intersection, namely the sequence $(n_k)$ is log-concave, see \cite{St}.
In higher dimensions, even the unimodality (and hence the Lefschetz property for $N(f)$) can fail, see
\cite[Remark 1.5]{St} for an example of this situation when $n=3$ provided by Aldo Conca. Hence the claim in Theorem \ref{mainL} does not seem to hold for $n>2$.}
  \end{Remark}

\end{document}